\documentclass[a4paper]{article}       

\usepackage{amsmath,mathtools,amsthm}
\usepackage{amssymb, authblk}
\usepackage[mathscr]{eucal}
\usepackage{bm}
\usepackage{bbm}
\usepackage[shortlabels]{enumitem}
\usepackage{hyperref}
\usepackage[usenames,dvipsnames]{xcolor}

\usepackage{tikz}
\usepackage{graphicx, caption,subcaption}
\usepackage{todonotes}

\newcommand{\prob}{\mathbb{P}}
\newcommand{\Prob}[1]{\prob\left(#1\right)}

\newcommand{\expec}{\mathbb{E}}
\newcommand{\Exp}[1]{\expec\left[#1\right]}

\newcommand*{\swap}[2]{#2#1}

\allowdisplaybreaks

\newtheorem{theorem}{Theorem}[section]

\newtheorem{lemma}[theorem]{Lemma}
\newtheorem{proposition}[theorem]{Proposition}
\newtheorem{corollary}[theorem]{Corollary}

\usepackage{a4wide}
\title{Scale-free graphs with many edges}
\author{Clara Stegehuis and Bert Zwart}
\date{\today}

\begin{document}
\maketitle
 \begin{abstract}
     We develop tail estimates for the number of edges in a Chung-Lu random graph with regularly varying weight distribution. Our results show that the most likely way to have an unusually large number of edges is through the presence of one or more hubs, i.e.\ edges with degree of order $n$.
 \end{abstract}
\section{Introduction and main results}

We analyze a sequence of random graphs introduced by~\cite{boguna2003,chung2002} which is constructed as follows. 
Let $n$ be the number of vertices and let $X_i, i\geq 1$, be an i.i.d.\ sequence of non-negative random variables with mean $\mu$ and a right tail which is regularly varying with index $\alpha>1$: 
\begin{equation}
    \label{xtail}
    \prob(X_1>x)=L(x)x^{-\alpha},
\end{equation}
for $x>0$, with $L(yx)/L(x)\rightarrow 1$ for $y>0$ as $x\rightarrow\infty$. 
$X_i$ can be interpreted as a weight for vertex $i$, and we denote $\mu =\Exp{X_i}$. A vertex with a high weight tends to have more edges: 
the probability $p_{ij}$ that an edge is present between vertices $i$ and $j$ equals
\begin{equation}\label{eq:pij}
    p_{ij}=p_{ij}^n(X_i,X_j) := \min \left\{ \frac{X_iX_j}{\mu n},1 \right\}.
\end{equation}

Given i.i.d.\ uniform $[0,1]$ random variables $U_{ij}, 1\leq i< j\leq n$, we define the total number of edges $E_n$ in the graph as 
\begin{equation}\label{eq:en}
    E_n := \sum_{1\leq  i< j\leq n}^n \mathbb{1}(U_{ij} \leq \min \{ X_iX_j/(\mu n),1 \}),
\end{equation}
where $\mathbb{1}$ denotes the indicator function. 
The mean of $E_n$ grows as $\mu n/2$. 
The specific purpose of this study is to investigate the probability that $E_n$ has significantly more edges than usual, i.e.$$\prob(E_n > (\mu/2+a) n)$$ for some fixed $a>0$. Our broader aim is to contribute to a better understanding of large-deviations properties of random graphs with power-law degrees.
In the past decade there has been  increased activity in establishing large deviations for random graphs. There now exist various large-deviations results for dense graphs and sparse graphs with light-tailed degrees 
\cite{Bordenave, Chatterjeebook,chatterjee2020localization, dokuamponsah2010, liu2020}, 
which do not cover scale-free graphs.
The typical behavior of scale-free graphs is subject to intense research activity \cite{Remcobook, Remcobook2}, while their large-deviations analysis is so far restricted to the Pagerank functional \cite{Litvak2017, Mariana2021} or the cluster sizes for critical random graphs~\cite{hofstad2018}.


To describe our main results, we introduce additional notation. Denote the mean $M_n$ of $E_n$, conditional on the weights $X_1,\ldots,X_n$ by
\begin{equation}
    M_n := \sum_{1\leq i< j\leq n} \min \{ X_iX_j/(\mu n),1 \},
\end{equation}
and set $S_n=\mu n M_n$, i.e.\ 
\begin{equation}
    S_n := \sum_{1\leq i< j\leq n} \min \{ X_iX_j,\mu n \}.
\end{equation}
We now give a description of our main results. A key parameter is  
\begin{equation} 
k(a):= \lceil a/\mu \rceil.
\end{equation}
Assuming that $a/\mu$ is not an integer, we show that the most likely way for $S_n$ to reach a value exceeding $(\mu^2/2+a)n$
is by $k$ large 
(of order $n$) values of $X_i$, an event which has probability of order $n^k\prob(X_1>n)^k$. In particular, if $X_1,\ldots,X_k$ equal $a_1n,\ldots,a_kn$, the remaining $X_i, i>k$ 
have a typical value, and $k\ll n$ is fixed, then, invoking the weak law of large numbers yields
\begin{align}
 S_n & = \sum_{1\leq i<j\leq k}(\mu n)+  \sum_{1\leq i\leq k<j} \min \{ a_i n X_j,\mu n \} + \sum_{k< i<j\leq n} \min \{ X_i X_j,\mu n \}\nonumber\\
 & \approx o(n^2)+ n^2 \sum_{1\leq i\leq k}  \Exp{\min \{ a_i  X_{k+1},\mu  \} }+  (\mu n)^2/2.
 \end{align}
Following the intuition from large deviations for heavy-tailed random variables (see e.g.\ \cite{RBZ})
we need to choose $k$ as the smallest number such that there exist constants $a_1,\ldots,a_k$ to get $ \sum_{i=1}^k \Exp{\min \{a_iX_1, \mu \}}> a$. This leads to the choice $k=k(a)$. A transition in the number of required hubs in $a$ appears when $k(a)$ is integer, which then also changes the scaling of $n^k\prob(X_1>n)^k$. Precisely at this transition point, it is therefore difficult to obtain precise statements, which is why we will work with the assumption that $k(a)$ is non-integer. A more technical discussion on this topic can be found at the end of Section \ref{sec-Sn}.
To state our results formally, we define 
\begin{equation}\label{eq:C}
C(a_1,\ldots,a_k) :=    \sum_{i=1}^k \expec[\min \{a_iX_{k+1}, \mu \}]
\end{equation}
and we let, for $b>0$, $X_i^b, i\geq 1$, be an i.i.d.\ sequence such that $\prob(X_i^b>x) = (x/b)^{-\alpha}, x \geq b$.  Informally, the distribution of $X_i^b$ is that of $X_i$ conditioned on the event $\{X_i\geq b\}$, when the slowly-varying function $L(x)=1$. In our context, it emerges as the limit of $\prob(X_i/bn > x \mid X_i > bn)$ as $n\rightarrow\infty$.
Set $\eta(a)$ as the smallest value $\eta$ for which  $(k(a)-1)\mu + \Exp{\min \{\eta X_1, \mu \}}\geq a$. 
Observe that $\eta(a)>0$, and also $\eta(a)<\infty$ if $a/\mu$ is not an integer. Define the constant
\begin{equation}\label{eq:Ka}
    K(a) := \eta(a)^{-k(a) \alpha} \Prob{ C\big(X_1^{\eta(a)}, \ldots, X_{k(a)}^{\eta(a)}\big) \geq a}.
\end{equation}
We first state our main result on $S_n$. With $f(n)\sim g(n)$ we denote that the ratio of $f$ and $g$ converges to 1 as $n\rightarrow\infty$.

\begin{proposition}
\label{prop-Sn}
Assume that $a/\mu$ is not an integer. Then
\begin{equation}
    \prob(S_n > (\mu^2/2+a)n^2)  \sim  K(a)  (n\prob(X_1>n))^{k(a)}.
\end{equation}
\end{proposition}
 $S_n$ only involves randomness from the vertex weights $X_i$, while $E_n$ also involves randomness from the uniform random variables in~\eqref{eq:en}. Our main result, derived from Proposition~\ref{prop-Sn}, shows that the tail of $E_n$ behaves the same as the one of $M_n$:
\begin{theorem}
\label{thm-En}
Suppose that $a$ is not an integer. Then
\begin{equation}
  \prob(E_n > (\mu/2+a)n) \sim \prob(M_n > (\mu/2+a)n) = \prob(S_n> (\mu^2/2 +\mu a )n^2).
\end{equation}
Therefore, $\prob(E_n > (\mu/2+a)n)$ is regularly varying of index $-\lceil a \rceil (\alpha-1)$. In particular, 
\begin{equation}\label{eq:endeviations}
    \prob(E_n > (\mu/2+a)n)  \sim  K(\mu a) 
    (n\prob(X_1>n))^{\lceil a \rceil}.
\end{equation}
\end{theorem}
The intuition behind this result is similar to the intuition given for $S_n$, combined with the insight that the additional randomness generated by the uniform random variables $U_{ij}$ is of lesser importance: the event that the number of edges exceeds 
$(\mu/2+a)n$ is caused by $k=\lceil a \rceil$ hubs, i.e.\ vertices with nodes of weight of order $n$. 
More in particular, our proofs give the insight that the weights of the $k$ hubs, normalized by $n$, converge weakly to $(X_1^{\eta (\mu a)},\ldots,X_k^{\eta(\mu a)})$ conditioned upon 
$C(X_1^{\eta(\mu a)},\ldots,X_k^{\eta(\mu a)}) \geq \mu a$ as $n\rightarrow\infty$.


To prove Theorem \ref{thm-En}, 
we use well-known concentration bounds for non-identically distributed Bernoulli random variables to show that $E_n$ and $M_n$ are close, facilitated by an estimate for the lower tail of $S_n$.
It is difficult to get rid of the integrality condition in  Theorem \ref{thm-En}, as this is where a transition occurs between the number of hubs that are needed. We are able to derive a weaker result, namely a large-deviations principle. Define $I(x) = (\alpha-1) \lceil x\rceil$ if $x\geq 0$ and $\infty$ otherwise. Although $I$ is discontinuous on its effective domain, it is lower semi-continuous, so that $I$ is a rate function. Define $\widehat{E}_n= E_n/n-\mu/2$.
\begin{corollary}
\label{cor-ldp}
$\widehat{E}_n,n\geq 1$, satisfies a large-deviations principle with speed $\log n$ and rate function $I$, i.e. for any Borel set $A$,
\begin{equation}
\label{ldp}
-\inf_{x\in \mathring{A}} I(x)\leq \liminf_{n\rightarrow\infty} \frac{\log \prob(\widehat{E}_n \in A)}{\log n} \leq \limsup_{n\rightarrow\infty} \frac{\log \prob(\widehat{E}_n \in A)}{\log n} \leq -\inf_{x\in \overline{A}} I(x).
\end{equation}
\end{corollary}
Our results constitute another case where a rare event in the presence of heavy tails is caused by multiple big jumps. Other heavy-tailed systems exhibiting rare events  with multiple big jumps are exit problems \cite{Ayan2022}, fluid networks \cite{Chen2019, zwart2004exact}, multi-server queues \cite{bazhba2019queue, foss2006heavy, foss2012large}, and reinsurance problems \cite{Albrecher}. For sample-path large deviations of heavy-tailed random walks, see \cite{RBZ}. 

In all our asymptotic results, the slowly varying function $L(x)$ plays no essential role, our techniques essentially allow us to treat the case of a general slowly varying function 
without any significant additional effort. The probability of a hub of weight at least $\varepsilon n$ is dominated by the power-law part of the distribution. However, $L(x)$ is included implicitly in our results, for example in $\Prob{X_1>n}$, but also in the definition of $C$ in~\eqref{eq:C} in Proposition~\ref{prop-Sn}.

The rest of this article is organized as follows. In Section \ref{sec-lemmas} we gather some preliminary results from the literature needed for our proofs. 
The proof of Proposition \ref{prop-Sn} is developed in Section \ref{sec-Sn}. The proof of Theorem \ref{thm-En} is presented in Section \ref{sec-En}. The proof of Corollary \ref{ldp} is given in Section \ref{sec-ldp}, and we end with a short discussion of our results.

\section{Preliminary results} 
\label{sec-lemmas}

The following lemma is a key estimate for sums of truncated heavy-tailed random variables, which is a reformulation of
Lemma 3 in \cite{resnick1999}.

\begin{lemma}
\label{lem-truncated}
For every $\delta>0$ and $\beta<\infty$ there exists an $\varepsilon>0$ such that
\begin{equation}
    \Prob{\sum_{i=1}^nX_i> (\mu+\delta)n, X_i \leq \varepsilon n, i=1,\ldots,n} = o(n^{-\beta}). 
\end{equation}
\end{lemma}

We proceed by stating a version of Chernoff's bound for sums of independent Bernoulli random variables. The statement is a variation of 
Theorem A.1.4 in \cite{AlonSpencer}.

\begin{lemma}
\label{lem-concentration}
Let $B_i,i\geq 1$ be a sequence of independent Bernoulli random variables with $p_i=\prob(B_i=1)=1-\prob(B_i=0)$.
Set $\mu_n = \sum_{i=1}^n p_i$. For every $b>0$ we have 
\begin{equation}
    \Prob{\sum_{i=1}^n B_i > (1+b) \mu_n} \leq e^{-\mu_n I_B(b)}, \hspace{1cm} \Prob{\sum_{i=1}^n B_i < (1-b) \mu_n} \leq e^{-\mu_n I_B(-b)},
\end{equation}
with $I_B(b) = (1+b) \log (1+b)-b$.
\end{lemma}

We finally state an elementary tail bound for binomially distributed random variables.

\begin{lemma}
\label{lem-binomial}
Suppose $B(n,p)$ has a binomial distribution with parameters $n$ and $p$. Then 
\begin{equation}
    \prob(B(n,p) \geq m ) \leq (np)^m.
\end{equation}
\end{lemma}

\begin{proof}
Set $B(n,p) = \sum_{i=1}^n B_i$, note that $\Prob{ B(n,p)\ge m} = \Prob{\exists i_1,\ldots, i_m \text{ s.t. } B_{i_j}=1}$ and apply the union bound. 
\end{proof}

\section{Proof of Proposition \ref{prop-Sn}}
\label{sec-Sn}

Throughout this section, we fix $a$ such that  $a/\mu$ is not an integer and write $k(a)=k, \eta(a)=\eta$.
Define for $\varepsilon>0$:
\begin{equation}
    N_{n,\varepsilon} := |\{ i \leq n: X_i > \varepsilon n \}|.
\end{equation}
The idea of the proof is to subsequently rule out the events $N_{n,\varepsilon}<k$ and $N_{n,\varepsilon}>k$. After that, we condition on $N_{n,\varepsilon}=k$ to work out the remaining technical details. This will be the focus of the next three lemmas which together form the proof of Proposition \ref{prop-Sn}. 

\begin{lemma}
There exists $\varepsilon>0$ such that
\begin{equation}
    \prob(S_n > (\mu^2/2+a)n^2;N_{n,\varepsilon} \leq k-1 ) = o \Big(\Big(n\prob(X_1>n)\Big)^k\Big). 
\end{equation}
\end{lemma}

\begin{proof}
We prove this lemma by suitably upper bounding $S_n$ to invoke Lemma \ref{lem-truncated}. Let $m\leq k$.
Set for fixed $\varepsilon>0$ the event 
\begin{equation}\label{eq:Am}
A_m:= \{X_i>\varepsilon n, i<m;  X_i \leq \varepsilon n, i\geq m \}.
\end{equation}
Write 
\begin{equation}
\label{comb}
     \prob(S_n > (\mu^2/2+a)n^2;N_{n,\varepsilon} =m-1 ) =  \binom{n}{m-1} \prob(S_n > (\mu^2/2+a)n^2; A_m).
\end{equation}
On the event $A_m$,
\begin{align*}
S_n  &\leq  \mu n (m-1)(m-2)/2 + \sum_{i,j\geq m ,i< j} \min \{X_iX_j, \mu n\}  + \sum_{i<m, j\geq m} \min \{X_i X_j, \mu n\}\\
&\leq \tfrac12\mu n (m-1)^2 + \tfrac12\Big(\sum_{i\geq m}X_i\Big)^2 + (m-1)n^2\mu.
\end{align*}
Thus, 
\begin{align}
\label{ubsn1}
     & \Prob{S_n > \Big(\frac{\mu^2}{2}+a\Big)n^2; A_m}\nonumber\\
     & \leq \prob\Bigg(\sum_{i\geq m}X_i  > \sqrt{ \mu^2n^2 +2(a- (m-1)\mu)n^2 - \mu n (m-1)^2} ; A_m\Bigg).
\end{align}
Recalling that $k= \lceil a/\mu \rceil$, we obtain that $a/\mu>k-1\geq m-1$, and therefore 
$(a- (m-1)\mu)> 0$. Consequently, there exists a $\zeta>0$ such that for sufficiently large $n$, 
$$\sqrt{ \mu^2n^2 +2(a- (m-1)\mu)n^2 - \mu n (m-1)^2} > (\mu+\zeta)n.$$
We can now bound (\ref{ubsn1}) for $n$ large, by 
\begin{equation}
     \prob(S_n > (\mu^2/2+a)n^2; A_m) \leq \prob\Bigg(\sum_{i\geq m}X_i  > (\mu+\zeta)n; A_m\Bigg) = o\Big( \prob(X_1>n)^k\Big),
\end{equation}
for suitably small $\varepsilon$, where we have applied Lemma \ref{lem-truncated} in the last equality. 
Invoking (\ref{comb}) and summing the estimates over  $m=1,\ldots,k$ gives the desired result. 
\end{proof}


\begin{lemma}
There exists $\varepsilon>0$ such that
\begin{equation}
    \prob(S_n > (\mu^2/2+a)n^2;N_{n,\varepsilon} \geq k+1 ) = o ((n\prob(X_1>n))^k). 
\end{equation}
\end{lemma}

\begin{proof}
We observe that $N_{n,\varepsilon}$ has a Binomial distribution with parameters $n$ and $\prob(X_1>\varepsilon n)$ and invoke Lemma \ref{lem-binomial}:
\[
\prob(S_n > (\mu^2/2+a)n^2;N_{n,\varepsilon} \geq k+1 ) \leq  \prob(N_{n,\varepsilon} \geq k+1 )  \leq 
(n\prob(X_1>n))^{k+1},
\]
which is $o ((n\prob(X_1>n))^k)$ by~\eqref{xtail}. \end{proof}

We are left to consider  $\prob(S_n > (\mu^2/2+a)n^2;N_{n,\varepsilon} = k)$.
Recall that $\eta$ is the smallest value such that  $ (k-1)\mu + \Exp{\min \{\eta X_1, \mu \}} \geq a$.
\begin{lemma}
\begin{equation}
    \lim_{n\rightarrow\infty} \frac{\prob(S_n > (\mu^2/2+a)n^2;N_{n,\varepsilon} = k)}{(n\prob(X_1>n))^k}= \eta^{-k \alpha} \prob( C(X_1^\eta, \ldots, X_k^\eta) \geq a). 
\end{equation}
\end{lemma}
\begin{proof}
Write
\begin{equation}
\label{lemma33eq1}
    \prob(S_n > (\mu^2/2+a)n^2;N_{n,\varepsilon} = k) = \binom{n}{k} \prob(S_n > (\mu^2/2+a)n^2 ; A_{k+1}),
\end{equation}
with $A_{k+1}$ as in~\eqref{eq:Am}.
To analyze $\prob(S_n > (\mu^2/2+a)n^2 ; A_{k+1})$, 
define the random variable $S_n(x_1,\ldots,x_k)$ as $S_n$ conditioned on $X_i=x_in, i=1,\ldots,k$, but where $X_i$ for $i>k$ are random and distributed as~\eqref{xtail}. More precisely, 
\begin{align}
    S_n(x_1,\ldots,x_k) := &  \sum_{1\leq  i< j\leq k}  \min \{ x_ix_j n^2,\mu n \}+  \sum_{1\leq  i\leq k < j} \min \{ x_inX_j ,\mu n \}\nonumber\\
    & +  \sum_{k< i< j} \min \{ X_iX_j,\mu n \}.
\end{align}
Recall that $C(a_1,\ldots,a_k)=    \sum_{i=1}^k \expec[\min \{a_iX_{k+1}, \mu \}]$.
From the weak law of large numbers, it follows that $S_n(x_1,\ldots,x_k)/n^2\rightarrow 0+ C(x_1,\ldots,x_k) + \mu^2/2$.
Consequently,
\begin{align}
\label{Snklimit}
  \prob(  S_n(x_1,\ldots,x_k) > (\mu^2/2+a)n^2; A_{k+1}) &\rightarrow \begin{cases}1, &  C(x_1,\ldots,x_k)> a,\\
  0, & C(x_1,\ldots,x_k) < a.\end{cases} 
\end{align}
Next, recall that $X_i^\varepsilon, i\geq 1$ are i.i.d.\ random variables with support on $[\varepsilon,\infty)$ such that $\prob(X_i^\varepsilon>y) = (y/\varepsilon)^{-\alpha}$.
Write $\prob(S_n > (\mu^2/2+a)n^2 ; A_{k+1})$ as 
\begin{align}
    &  \int_{(\varepsilon, \infty)^k} \prob(  S_n(x_1,\ldots,x_k) > (\mu^2/2+a)n^2 ; X_i<\varepsilon n, i>k) d\prod_{i=1}^k \Prob{\frac{X_i}{n} \leq x_i \mid X_i >\varepsilon n}\nonumber\\
    & \quad \times \prob(X_1>\varepsilon n)^k .
\end{align}
Since $\prob(X_i/n \leq x_i \mid X_i >\varepsilon n)$ converges to the continuous distribution $\prob(X_i^\varepsilon\leq x_i)$, 
we can ignore the contribution to the integral of all $(x_1,\ldots,x_k)$ such that $C(x_1,\ldots,x_k)=a$. To see this, 
note that $a\rightarrow \expec[\min \{aX_{k+1}, \mu \}]$ is strictly increasing on $[0,1/u^*)$, with 
$u^*=\inf \{ u>0: \Prob{X_{k+1}>u}=1 \}$. If $u^*>0$, then the subset of $\mathbb{R}^k$ on which $C(x_1,\ldots,x_k)$ is constant is contained in $[1/u^*, \infty)^k$. If $x_i \geq 1/u^*$ for each $i$, then
$C(x_1,\ldots,x_k) = \mu k = \mu \lceil a/\mu\rceil > a$, since $a/\mu$ is not an integer. Consequently, the set of $(x_1,\ldots,x_k)$ for which $C(x_1,\ldots,x_k)=a$ has Lebesgue measure 0 in $\mathbb{R}^k$.
Thus, we can apply \eqref{Snklimit} to show that the integral in the last display converges to $\prob( C(X_1^\varepsilon, \ldots, X_k^\varepsilon) \geq a)$. This probability is strictly positive, since $a/\mu$ is non-integer.
To eliminate the auxiliary parameter $\varepsilon$, note that $C(x_1,\ldots,x_k)<a$ as soon as there exists some $i$ such that $x_i<\eta$, as the other terms contribute at most $(k-1)\mu$ to the summation.
Therefore, if $\varepsilon<\eta$,
\begin{equation}
\label{eq-identity-K}
\prob( C(X_1^\varepsilon, \ldots, X_k^\varepsilon) \geq a) = (\eta/\varepsilon)^{-k\alpha}\prob( C(X_1^\eta, \ldots, X_k^\eta) \geq a).
\end{equation}
Furthermore, by regular variation, 
\begin{equation}
    \prob(X_1>\varepsilon n)^k \sim (\eta/\varepsilon)^{k\alpha}  \prob(X_1>\eta n)^k.
\end{equation}
Putting everything together, we conclude that
\begin{equation}
    \prob(S_n > (\mu^2/2+a)n^2 ; A_{k+1}) \sim \prob( C(X_1^\eta, \ldots, X_k^\eta) \geq a) \prob(X_i>\eta n)^k.
\end{equation}
The lemma now follows from (\ref{lemma33eq1}) and the fact that $\binom{n}{k}\sim n^k$.
\end{proof}

We close this section with some technical comments on the integrality condition on $a/\mu$. From the heuristics given so far, it is clear that it helps to distinguish between scenarios involving $k(a)$
or $k(a)+1$ jumps, and it guarantees that the pre-factor $\prob( C(X_1^\eta, \ldots, X_k^\eta) \geq a)$ is strictly positive. If $a/\mu$ is integer and 
$u^*=\inf \{ u>0: \Prob{X_{k+1}>u}=1 \}=0$, then $\prob( C(X_1^\eta, \ldots, X_k^\eta) \geq a)=0$ and the true asymptotics will change. We conjecture 
that either $a/\mu$ or $a/\mu+1$ hubs are needed.
If $u^*>0$ and $a/\mu$ is integer, then $\prob( C(X_1^\eta, \ldots, X_k^\eta) \geq a)>0$. In that case we expect that the dominant scenario is $a/\mu$ hubs, but the above proof method breaks down, and we need to understand (at least) the second-order properties
of $S_n(x_1,\ldots,x_k)$ as $n\rightarrow\infty$. To develop a complete understanding in each of the cases $u^*>0$ and $u^*=0$ requires methods which are beyond the scope of this study. 
For example one may obtain a central limit theorem for the number of edges 
by extending some of the results in \cite{kifer} to take into account truncation, as is done for sums of truncated i.i.d.\ heavy-tailed random variables in
\cite{chakrabarty}. Also in the next section, the non-integrality assumption plays an important role. 

\section{Proof of Theorem \ref{thm-En}}
\label{sec-En}

The proof of Theorem \ref{thm-En} is based on suitably bounding the difference between $E_n$ and its conditional mean $M_n=S_n/\mu n$, using the concentration
bounds in Lemma \ref{lem-concentration}. For this procedure to work, we need an asymptotic estimate for the lower tail of $S_n$. 
Since $X_i, i\geq 1$, are non-negative random variables, this estimate is considerably easier  to obtain than the upper tail. 

\begin{lemma}
\label{lem-lowertailsn}
For each $a>0$, there exists a $\delta>0$ such that
\begin{equation}
\label{eq:lowertailsn}
    \prob(S_n \leq  (\mu^2/2-a)n^2) = O(e^{-\delta n}).
\end{equation}
\end{lemma}

\begin{proof}
Define $X_i^M=\min \{X_i, M\}$. Let $M<\infty$ be large enough such that \linebreak $\expec[\min \{X_i, M\}]^2 \geq \mu^2 - a/2$.
Observe that, for sufficiently large $n$,
\begin{equation}
    S_n= \sum_{1\leq i< j\leq n} \min \{X_iX_j, \mu n\} \geq  \sum_{1\leq i< j\leq n} X_i^M X_j^M
\geq  \tfrac12\Big(\sum_{i=1}^n X_i^M\Big)^2 - nM^2.
\end{equation}
The estimate (\ref{eq:lowertailsn}) 
now follows by an application of Chernoff's bound to $\sum_{i=1}^n -X_i^M$. 
\end{proof}


\begin{proof}[Proof of Theorem \ref{thm-En}]
Conditional on $X_1,\ldots,X_n$, the variables $B_{ij}, i< j$, indicating whether there is an edge between node $i$ and $j$, are independent. 
Therefore, observing $M_n = S_n/(n\mu) = \expec[E_n \mid X_1,\ldots, X_n]$,
we can apply Lemma \ref{lem-concentration} to obtain that, for $b>0$,
\begin{equation}
      \prob( |E_n - M_n| > b M_n \mid X_1,\ldots,X_n) \leq 2e^{-M_n J(b)}
\end{equation}
almost surely, with $J(b) = \min \{I_B(b), I_B(-b)\}$. 
Now, write for fixed $\varepsilon>0$, 
\begin{align}
\label{twoterms}
    \prob(E_n> (\mu/2+a)n)  & =   \prob(E_n> (\mu/2+a)n; |E_n - M_n| \leq \varepsilon M_n) \nonumber\\
    & \quad +  \prob(E_n> (\mu/2+a)n ;  |E_n - M_n| >\varepsilon M_n).
\end{align}
Invoking Lemma \ref{lem-lowertailsn}, the second term on the RHS of (\ref{twoterms}) is smaller than
\begin{align}
\label{bound-EnEn}
 \prob(|E_n - M_n| >\varepsilon M_n) &\leq
    \prob( |E_n - M_n| >\varepsilon M_n ;M_n > \zeta n) + \prob(M_n \leq \zeta n) \nonumber\\
    & \leq 2e^{-\zeta n J(b)} + O(e^{-\delta n})
\end{align}
for some $\delta>0$ depending on $\zeta>0$, the latter chosen suitably small. 
We conclude that (making $\delta$ smaller than $\zeta J(b)$ if needed) 
\begin{align}
\label{keyidentity}
    \prob(E_n> (\mu/2+a)n)&  =   \prob(E_n> (\mu/2+a)n; |E_n - M_n| \leq \varepsilon M_n) + O(e^{-\delta n})\nonumber\\
    & \leq \prob(M_n> (\mu/2+a-\varepsilon)n) + O(e^{-\delta n})
    .
\end{align}
We use this identity to prove asymptotic lower and upper bounds which together complete the proof of Theorem \ref{thm-En}.
Invoking (\ref{keyidentity}) and Proposition \ref{prop-Sn} for $M_n=S_n/(\mu n)$, we see that
\begin{align*}
    &\limsup_{n\rightarrow\infty} \frac{\prob(E_n> (\mu/2+a)n)}{\prob(M_n> (\mu/2+a)n)} \leq \limsup_{n\rightarrow\infty} \frac{\prob(M_n> (\mu/2+a-\varepsilon )n)}{\prob(M_n> (\mu/2+a)n)}\\ 
    &= 
   \frac{ \eta(\mu(a-\varepsilon))^{-k(\mu a) \alpha} \Prob{ C\Big(X_1^{\eta(\mu(a-\varepsilon))}, \ldots, X_{k(\mu a)}^{\eta(\mu(a-\varepsilon))} \Big)\geq \mu(a-\varepsilon) }}{ \eta(\mu a)^{-k(\mu a) \alpha} \Prob{ C\Big(X_1^{\eta(\mu a)}, \ldots, X_{k(\mu a)}^{\eta(\mu a)}\Big) \geq \mu a} }.
\end{align*}
In the last equality we have used that $k(\mu a)=k(\mu (a-\varepsilon))$, which holds because $a$ is non-integer. This property also implies that the last expression converges to $1$ if $\varepsilon\downarrow 0$, providing the upper bound.
The lower bound uses that
\begin{align}
     & \prob(E_n> (\mu/2+a)n; |E_n - M_n| <\varepsilon M_n) \nonumber\\
     & \quad \geq  \prob(M_n> (\mu/2+a+\varepsilon)n) - \prob(|E_n - M_n| >\varepsilon M_n)  .
\end{align}
The second term of the RHS is exponentially small in $n$, as shown in (\ref{bound-EnEn}). Consequently, invoking Proposition \ref{prop-Sn},
\begin{align*}
    &\liminf_{n\rightarrow\infty} \frac{\prob(E_n> (\mu/2+a)n)}{\prob(M_n> (\mu/2+a)n)} \geq \liminf_{n\rightarrow\infty} \frac{\prob(M_n> (\mu/2+a+\varepsilon )n)}{\prob(M_n> (\mu/2+a)n)}\\
    &= 
    \frac{ \eta(\mu(a+\varepsilon))^{-k(\mu a) \alpha} \Prob{ C\Big(X_1^{\eta(\mu(a+\varepsilon))}, \ldots, X_{k(\mu a)}^{\eta(\mu(a+\varepsilon))}\Big) \geq \mu(a+\varepsilon)} }{ \eta(\mu a)^{-k(\mu a) \alpha} \Prob{ C\Big(X_1^{\eta(\mu a)}, \ldots, X_{k(\mu a)}^{\eta(\mu a)}\Big) \geq \mu a} }.
\end{align*}
In the last equality we used that $k(\mu a)=k(\mu (a+\varepsilon))$, which holds because $a$ is non-integer, implying that the last expression converges to $1$ if $\varepsilon\downarrow 0$, providing the lower bound.
\end{proof}

\section{Proof of Corollary \ref{cor-ldp}}
\label{sec-ldp}

As a first step we show that the left tail of $E_n$ is lighter than polynomial.

\begin{lemma}
\label{lem-lowertailen}
For each $a>0$, there exists a $\delta>0$ such that
\begin{equation}
\label{lowertailsn}
    \prob(E_n \leq  (\mu/2-a)n) = O(e^{-\delta n}).
\end{equation}
\end{lemma}

\begin{proof}

Without loss of generality, we can assume $a<\mu/2$. Note that $\prob(E_n \leq  (\mu/2-a)n)$ can be upper bounded by
\begin{equation}
\label{eq-twoterms5}
     \int_{\mu/2-a/2}^\infty \prob(E_n \leq  (\mu/2-a)n \mid M_n = ny) d\prob(M_n/n\leq y)   + \prob(M_n \leq (\mu/2-a/2)n).
\end{equation}
The second term is exponentially small in $n$ due to Lemma \ref{lem-lowertailsn}. To analyze the first term, note that $E_n$ is a sum of Bernoulli variables with mean $M_n$. Thus, by conditioning on $M_n$, we can we can apply Lemma \ref{lem-concentration} to obtain 
\begin{align*}
    \prob(E_n \leq  (\mu/2-a)n \mid M_n = ny) &=   \Prob{E_n \leq  ny\Big(1 -   \frac{y-(\mu/2-a)}{y}\Big) \mid M_n = ny}\\
    &\leq e^{-ny I_B(-(1-(\mu/2-a)/y))} \leq e^{-n(\mu/2-a/2)I_B(-a/(\mu-a))}.
\end{align*}
The second inequality follows by noting that  $I_B$ is non-negative, strictly convex, and $0$ at $0$. 
Therefore, $y I_B(-(1-(\mu/2-a)/y))$ is increasing on $[\mu/2-a/2,\infty)$, so that we obtain the second inequality by replacing $y$ with $\mu/2+a/2$. 
Consequently, 
\[
\int_{\mu/2-a/2}^\infty \prob(E_n \leq  (\mu/2-a)n \mid M_n = ny) d\prob(M_n/n\leq y) \leq e^{-n(\mu/2-a/2)I_B(-a/(\mu-a))}.
\]
We have shown that both terms in (\ref{eq-twoterms5}) are exponentially small in $n$, completing the proof. 
\end{proof}

\begin{proof}[Proof of Corollary \ref{cor-ldp}]
Consider first $A$ closed. If $0\in A$, the upper bound is trivial. If $0\not\in A$ we can write $A = A_- \cup A_+$, with $a_-=\sup A_-<0$ and $a_+ =\inf A_+>0$. 
Since $A$ is closed and $0\not\in A$, both  $a_-$ and $a_+$ are elements of $A$, and $a_-<0<a_+$. Next, note that 
\[
\prob(\widehat{E}_n \in A) \leq \prob(\widehat{E}_n \leq a_-) +  \prob(\widehat{E}_n \geq a_+).
\]
Invoking Lemma \ref{lem-lowertailen}, the first term is exponentially small in $n$. By Theorem \ref{thm-En}, the second term is regularly varying with exponent $(\alpha-1) \lceil a_+ \rceil$ if $a_+$ is not an integer. If  $a_+$ is an integer, we can make $a_+$ a bit smaller, while keeping  $\lceil a_+ \rceil$  fixed, preserving the upper bound for $\prob(\widehat{E}_n \in A)$. 
This yields, using that $\log L(n)/ \log n\rightarrow 0$, and abbreviating the $n$-independent constant in Theorem \ref{thm-En} with $a=a_+$ by $K$,
\begin{align}\label{eq:aplus}
    \limsup_{n\rightarrow\infty} \frac{\log \prob(\widehat{E}_n \in A)}{\log n} &\leq 
    \limsup_{n\rightarrow\infty} \frac{\log [\prob(\widehat{E}_n \leq a_-)+ \prob(\widehat{E}_n \geq a_+)]}{\log n}\nonumber\\
    &\leq \limsup_{n\rightarrow\infty} \frac{\log [ K(n\prob(X_1>n))^{\lceil a_+\rceil}]}{\log n}\nonumber\\
    &=-(\alpha-1) \lceil a_+ \rceil =    -\inf_{x\in {A}} I(x).
\end{align}
Assume now that $A$ is open. If $\sup A \leq 0$ the result is straightforward, so assume that $\sup A>0$. 
For every $\varepsilon>0$, we can pick the following subset of $A$: take $a$ such that $a\in A$;  and $\inf_{x\in A} I(x) \geq I(a)-\varepsilon$. 
Since $A$ is open, we may modify the constant $a$ slightly such that $a$ is non-integer. Next, take a sufficiently small constant $b$ such that the ball around $a$ with radius $b$ is in $A$, such that $a-b/2$ and $a+b/2$ are both non-integer, and $\lceil a-b/2\rceil =\lceil a+b/2\rceil =\lceil a\rceil$.
Now, observe that
\begin{align*}
     \prob(\widehat{E}_n \in A) &\geq \prob(\widehat{E}_n \in (a-b/2, a+b/2))\\
     &= \prob( E_n > n(\mu/2+a-b/2)) - \prob(E_n \geq n(\mu/2+ a+b/2))\\
     &\sim \big(K(\mu(a-b/2)) - K(\mu(a+b/2))\big) (n\Prob{X_1>n})^{\lceil a \rceil},
\end{align*}
with $K(\cdot)$ as in~\eqref{eq:Ka}.
Using \eqref{eq-identity-K} we can write for a fixed $\delta \in (0,\eta(a\mu))$,
\begin{equation}
    K(a\mu) =  \delta^{-\alpha \lceil a \rceil} \Prob{ C(X^\delta_{1}, \ldots, X_{\lceil a\rceil)}^{\delta}} \geq  \mu a).
\end{equation}
Since $\lceil a\rceil$ is constant in a neighborhood of $a$, we see that $K(a\mu)$ is strictly decreasing in a neighborhood of $a$.
Consequently, $K(\mu(a-b/2)) - K(\mu(a+b/2))>0$ and we can apply Theorem \ref{thm-En} to conclude that $\prob(\widehat{E}_n \in (a-b/2, a+b/2))$
is regularly varying with index $I(a)$.
  Therefore, since $\log L(n) / \log n\rightarrow 0$,
\begin{align}
    \liminf_{n\rightarrow\infty} \frac{\log \prob(\widehat{E}_n \in A)}{\log n}  &  \geq \liminf_{n\rightarrow\infty} \frac{\log \Big(\big(K(\mu(a-b/2)) - K(\mu(a+b/2))\big) \big(n\Prob{X_1>n}\big)^{\lceil a \rceil}\Big)}
    {\log n} \nonumber\\
    & = -I(a) \geq - \inf_{x\in A}I(x)+\varepsilon.
\end{align}
Letting $\varepsilon \downarrow 0$ completes the proof.\end{proof}

\section{Discussion}

In this paper, we have studied the probability of a large number of edges in a heavy-tailed random graph model. We show that the most likely way to obtain at least $an$ more edges than expected is by $k(a)$ hubs with weight of order $n$. 

While this paper focuses on the Chung-Lu version of the inhomogeneous random graph~\ref{eq:pij}, there is a wide class of connection probabilities $p_{ij}$ with similar properties~\cite{hofstad2017b}. We therefore believe that our results can be extended to other connection probabilities in this class. Some of these connection probabilities construct random graphs that are similar to the erased configuration model, or the uniform random graph, suggesting that large deviations of the number of edges behaves similarly in these models. 

In the random geometric graph on the other hand, large deviations of the number of edges are caused by one large clique, due to the geometric nature of the model~\cite{chatterjee2020localization}. We here show that power-law random graphs on the other hand, are more likely to contain a large amount of edges due to the presence of hubs. It would therefore be interesting to investigate large deviations of edge counts for models with both geometry and power-law degrees, such as the hyperbolic random graph~\cite{krioukov2010} or  geometric inhomogeneous random graphs~\cite{bringmann2015}. 

 While the number of edges is one of the simplest graph statistics, we believe that there is a much wider class of graph statistics where the randomness of the i.i.d.~weights does not play a role in the large deviations properties, similarly to Theorem~\ref{thm-En}. Proving this however will be more involved for more complex statistics however, as for properties that depend on more than one edge, dependencies between the presences of these edges arise, due to their random weights. For relatively simple statistics, such as triangle counts, this is possible through an exhaustive enumeration of different cases~\cite{stegehuis2023}, but a more comprehensive method for a wider class of statistics would be interesting to investigate in further research.

\bibliographystyle{abbrv}


\end{document}